\documentclass{article}
\usepackage{graphicx} 
\usepackage{algorithm,algorithmic}
\usepackage{amsmath}
\usepackage{amsthm}
\usepackage{amssymb}
\usepackage{color}
\usepackage{bm}
\usepackage{float}
\usepackage{hyperref}
\newtheorem{lemma}{Lemma}[section]
\newtheorem{theorem}{Theorem}[section]
\newtheorem{remark}{Remark}[section]
\newtheorem{example}{Example}[section]

\usepackage[textsize=tiny]{todonotes}

\title{Properties of Nonlinear GMRES Applied to the Preconditioned Richardson Iteration\thanks{Department of Mathematics, University of Houston, 3551 Cullen Blvd, Room 641, Houston, Texas 77204-3008, USA (\tt{yhe43@central.uh.edu}).}}
\author{Yunhui He}

\begin{document}

\maketitle

\begin{abstract}
In this work, we propose new variants of Anderson acceleration and nonlinear GMRES for general fixed-point iterations, based on modified least‑squares problems associated with the methods. To solve the underlying linear systems, we apply these new approaches to accelerate the preconditioned Richardson iteration. We establish connections between the proposed variants and both left- and right-preconditioned GMRES. In particular, we show that full NGMRES applied to the preconditioned Richardson iteration is equivalent to right-preconditioned GMRES, while full NGMRES equipped with the new least-squares formulation is equivalent to left-preconditioned GMRES. Furthermore, under certain conditions on the preconditioned coefficient matrix, an equivalence between windowed NGMRES with any depth and preconditioned GMRES. These theoretical results deepen our understanding of NGMRES for solving linear systems and clarify its relationship to classical preconditioned GMRES. Finally, we establish conditions for monotonicity of the various variants. Numerical results are presented to validate our theoretical findings.
\end{abstract}


\section{Introduction}

Anderson acceleration (AA) \cite{anderson1965iterative} and nonlinear GMRES (NGMRES) \cite{washio1997krylov,oosterlee2000krylov} are widely used to accelerate the performance of fixed-point iterations due to its high efficiency. The main idea of both methods is to use a linear combination of the previous iterates to construct a new approximation, where the combination coefficients are obtained via solving a small dimension of least-squares problem in each step. This makes AA and NGMRES be nonstationary iterations and leads to more efficient solvers.  Substantial effort has been devoted to establishing their convergence theory  \cite{walker2011anderson,he2025convergenceNG,toth2015convergence,ouyang2024descent,GreifHe25NGMRES,evans2020proof,sterck2021asymptotic,rebholz2023effect}, developing new variants \cite{he2025generalizedHeLeveque,feng2024convergence,he2025generalizedN,tang2024anderson,chen2022composite}, and extending their applications \cite{sterck2012nonlinear,sterck2013steepest,wang2021asymptotic,lipnikov2013anderson,bian2022anderson,pollock2019anderson}. In this work, we are interested in designing new variants of AA and NGMRES, and making connections with existing algorithms.
 
Consider solving 
\begin{equation*}
    g(x)=0,
\end{equation*}
with a fixed-point iteration $q(x)$.  

When $g(x)=0$ is a linear system, we assume that $g(x)=Ax-b$, that is,
\begin{equation}\label{eq:Axb}
    Ax=b,
\end{equation}
where $A\in\mathbb{R}^{n \times n}$ is invertible and $x ,b \in\mathbb{R}^{n}$. We consider the fixed-point iteration $q(x)$ given by the preconditioned Richardson iteration
\begin{equation}\label{eq:PR-FP}
    q(x)=x+P(b-Ax),
\end{equation}
where $P$ is called a preconditioner and is invertible. We adopt this formulation using $P$ (instead of the conventional $P^{-1}$ in \eqref{eq:PR-FP}) purely for notational simplicity.

Define
\begin{equation*}  
    r(x)=x-q(x)=P(Ax-b),
\end{equation*}
which we refer to {\bf preconditioned residual}. We define the {\bf classical residual} as $\bar{r}=Ax-b$ of solving $Ax=b$.  Note that when $P=I$, the preconditioned residual and the classical residual are the same.  Denote the iteration matrix of \eqref{eq:PR-FP} as $B=I-PA$. In the literature, when $P=I$, the properties and convergence analysis of AA and NGMRES are well-established.  Although applications of AA to the preconditioned Richardson method have been reported in the literature (see \cite{chen2023short}), to the best of our knowledge, the theoretical results for the case  $P\neq I$ in AA remain absent and no corresponding results have been established for NGMRES.  

In this work, we first investigate the properties of AA and NGMRES when applied to the preconditioned Richardson iteration for solving linear systems. Then, based on the derivation of NGMRES, we introduce new variants of AA and NGMRES by replacing the original least‑squares problem with a newly formulated one and study their properties. We further make connections with preconditioned GMRES algorithm. Our contributions are listed below.
\begin{itemize}
    \item We propose a new variant of NGMRES that minimizes the preconditioned residual when applied to the preconditioned Richardson iteration.
    \item We establish that applying full NGMRES to the preconditioned Richardson iteration yields an algorithm equivalent to right-preconditioned GMRES, whereas the newly developed full NGMRES formulation based on reformulated least-squares problems corresponds to left-preconditioned GMRES.
    
    \item Under suitable conditions on the preconditioned coefficient matrix, we demonstrate that windowed NGMRES($m$) with any depth $m$ is equivalent to preconditioned GMRES.
    \item We introduce two AA variants that employ new least-squares formulations. When applied to the preconditioned Richardson iteration, the two variants minimize the residual and the preconditioned residual, respectively.
\end{itemize}

Our theoretical results provide a deeper and more unified perspective on the NGMRES method for solving linear systems, clarifying its relationship to both left- and right-preconditioned GMRES. These findings not only strengthen the theoretical foundation of NGMRES but also yield alternative formulations of preconditioned GMRES that broaden the methodological landscape for iterative linear solvers. We also explore properties of new variants of AA. Such insights have the potential to inspire new algorithmic developments, inform the design of more effective preconditioners, and improve our overall understanding of accelerated fixed-point and Krylov-based methods. 

Although our theoretical results are developed in the context of linear problems, the new variants are formulated for use in general fixed-point iterations. We leave the study of the applications of these methods to nonlinear problems for future work.

The remainder of this work is organized as follows. In section \ref{sec:AA-NGMRES-P}, we derive some properties of AA and NGMRES applied to preconditioned Richardson iteration. In section \ref{sec:new}, we propose several variants of AA and NGMRES based on modified least-squares problems, and establish their relationships with preconditioned GMRES. In section \ref{sec:num}, we present numerical results that illustrate the effectiveness of these new approaches in solving linear problems. We draw conclusions in section \ref{sec:con}.
 
\section{AA and NGMRES on preconditioned Richardson iteration}\label{sec:AA-NGMRES-P}
Before presenting our new variants of AA and NGMRES, we begin with a brief review of the existing results for these two methods when applied to unpreconditioned Richardson iteration, then we extend some of the results to preconditioned cases.

For simplicity, we denote $\|\cdot \|$ as the 2‑norm for matrices and vectors. For sequences generated by AA, NGMRES or GMRES, we will, when necessary, append the superscript, $A, NG, G$ to distinguish them. For example, $x_k^A$ denotes the $k$th iterate generated by AA. If AA and NGMRES use all previous iterates information, i.e., $m_k=\min\{m,k\}=k$ for all $k$ in Algorithms \ref{alg:AA} and \ref{alg:NGMRES}, we refer to the corresponding methods as full AA and full NGMRES, respectively. Otherwise, we refer to them as windowed AA and windowed NGMRES.

\subsection{AA on preconditioned Richardson iteration}
We present the AA framework in Algorithm \ref{alg:AA}. The main idea of AA is that the update is a linear combination of several previous fixed-point iterates; see \eqref{eq:xkp1-AA}. To the best of our knowledge, no explanation has been provided for why the combination coefficients $\{\alpha_i^{(k)}\}$ in  \eqref{eq:xkp1-AA} are determined by \eqref{eq:min-AA}. For simplicity, let us denote $\bm{\alpha}^{(k)}=\big(\alpha_1^{(k)},\cdots, \alpha_{m_k}^{(k)}\big)$. Note that in Algorithm \ref{alg:AA}, if $m=0$, AA($m$) reduces to the underlying fixed-point iteration, $q(x)$. The choice of $m$ is often problem-dependent. However, it is usually relatively small, generally on the order of a few to several tens.
	\begin{algorithm}[H] 
		\caption{AA($m$):   Anderson acceleration with depth $m$} \label{alg:AA}
		\begin{algorithmic}[1] 
			\STATE  \textbf{input:} $x_0$  and $m\geq0$ with $m\in\mathbb{N}$
			\FOR {$k=0,1,\cdots$ until convergence }
			\STATE compute 
			\begin{equation}\label{eq:xkp1-AA} 
				x_{k+1} = q(x_k) + \sum_{i=1}^{m_k}\alpha_i^{(k)} \left(q(x_k)- q(x_{k-i}) \right),
			\end{equation}
			where $m_k=\min\{k,m\}$ and $\bm{\alpha}^{(k)}=\big(\alpha_1^{(k)},\cdots, \alpha_{m_k}^{(k)}\big)$ is obtained by solving the  least-squares problem
			\begin{equation}\label{eq:min-AA} 
				\min_{\bm{\alpha}^{(k)}} \left\|r(x_k)+\sum_{i=1}^{m_k} \alpha_i^{(k)} \left(r(x_k)-r(x_{k-i}) \right) \right\|^2,
			\end{equation}
            where $r_k=x_k-q(x_k)$.
			\ENDFOR
        \STATE  \textbf{output:} $x_{k+1}$ 
		\end{algorithmic}
	\end{algorithm}

In \cite{walker2011anderson}, it was shown that when full AA is applied to the unpreconditioned Richardson iteration $q(x_k)=x_k+(b-Ax_k)$, and both full AA and GMRES use the same initial guess to solve $Ax=b$, with the 2-norm of the residuals of GMRES strictly decreasing, then $x_{k+1}=q(x_k^G)=x_k^G-r_k^G$, where $x_{k+1}$ is the iterate generated by full NGMRES, $x_k^G$ is the iterate generated by GMRES, and $r_k^G=Ax_k^G-b$. In this situation, it is easily obtained $r_{k+1}=Br_k^G$, where $B=I-A$ is the iteration matrix of $q(x)$.  In \cite{GreifHe25NGMRES}, we have shown that when applying AA to the unpreconditioned Richardson iteration and $B$ is invertible, the least-squares problem defined by \eqref{eq:min-AA} can be rewritten as 
\begin{equation}\label{eq:AA-LSP-inverseG}
 \min_{{\bm \alpha}^{(k)}} \|B^{-1}r_{k+1}\|^2,
\end{equation}
where $r_{k+1}=x_{k+1}-q(x_{k+1})=Ax_{k+1}-b$. In general, there is no guaranty that the sequence $\{\|r_{k+1}\|\}$ of AA is decreasing.
However, for the full AA, if the 2-norm of the residuals of GMRES is strictly decreasing,  we have
\begin{equation*}
    \|B^{-1}r_{k+1}\|=\|r_k^G\|< \|r_{k-1}^G\|=\|B^{-1}r_{k}\|,
\end{equation*}
which means that the sequence $\{B^{-1}r_{k+1}\}$ is strictly decreasing.

In the following, we will show that \eqref{eq:AA-LSP-inverseG}  holds for the preconditioned Richardson iteration, \eqref{eq:PR-FP}, i.e, $B=I-PA$ and $r_{k+1}=P(Ax_{k+1}-b)$. We present properties of AA applied to the preconditioned Richardson iteration below.
\begin{theorem}\label{thm:AA-org-P}
 Consider AA($m$) presented in Algorithm \ref{alg:AA} to accelerate preconditioned Richardson iteration, \eqref{eq:PR-FP}. Let $B=I-PA$ and $r_{k+1}=x_{k+1}-q(x_{k+1})$. Assume that $B$ is invertible. Then, the least-squares problem defined by \eqref{eq:min-AA} can be rewritten as 
\begin{equation}\label{eq:AA-LSP-inverseG-P}
 \min_{{\bm \alpha}^{(k)}} \|r_k-R_k{\bm \alpha}^{(k)}\|^2=\min_{{\bm \alpha}^{(k)}} \|B^{-1}r_{k+1}\|^2,
\end{equation}
where 
\begin{equation}\label{eq:defRk}
R_k=[r_{k-1}-r_k, r_{k-2}-r_k, \cdots, r_{k-m_k}-r_k].
\end{equation}
In addition, for $k\geq 0$
\begin{equation*}
    \|B^{-1}r_{k+1}\|\leq\|r_k\|,
\end{equation*}
and for $i, j =0, 1, \cdots, m_k$, we have the following orthogonal property:
\begin{equation*}
 (B^{-1}r_{k+1})^T(r_{k-i}-r_{k-j})=0.
\end{equation*}
Furthermore, if $\|B\|<1$, we have
\begin{equation*}
    \|r_{k+1}\|<\|r_k\|.
\end{equation*}
\end{theorem}
\begin{proof} 
In AA($m$), from \eqref{eq:min-AA} we have
\begin{equation}\label{eq:AA-lsp-Pre}
 r(x_k)+\sum_{i=1}^{m_k} \alpha_i^{(k)} \left(r(x_k)- r(x_{k-i}) \right)  =r_k -\sum_{i=1}^{m_k} \alpha_i^{(k)} (r_{k-i}-r_k)
    = r_k- R_k\bm{\alpha}^{(k)}.
\end{equation}
In  AA($m$), from \eqref{eq:xkp1-AA} we have
\begin{align*}
    x_{k+1} =& q(x_k) + \sum_{i=1}^{m_k}\alpha_i^{(k)} \left(q(x_k)- q(x_{k-i}) \right) \\
    =&x_k+P(b-Ax_k) + \sum_{i=1}^{m_k}\alpha_i^{(k)} \left((x_k-x_{k-i}) -PA(x_k-x_{k-i})  \right)\\
    =& (I-PA)x_k + Pb+ \sum_{i=1}^{m_k}\alpha_i^{(k)} (I-PA)(x_k-x_{k-i})  \\
    =& (I-PA) \left(x_k + \sum_{i=1}^{m_k}\alpha_i^{(k)}(x_k-x_{k-i}) \right)+Pb.
\end{align*}
It follows that 
\begin{align}
     r(x_{k+1})&=x_{k+1}-q(x_{k+1})\nonumber\\
     &= P(Ax_{k+1} -b)\nonumber\\
     &=  PA\left(  (I-PA) \left(x_k + \sum_{i=0}^{m_k}\alpha_i^{(k)}(x_k-x_{k-i}) \right)+Pb \right)-Pb \nonumber\\
     & =(I-PA) \left((PAx_k-Pb) +\sum_{i=0}^{m_k} \alpha_i^{(k)} PA(x_k- x_{k-i}) \right)\nonumber\\
      & =(I-PA) \left(r_k -R_k\bm{\alpha}^{(k)}\right)\nonumber\\
      &=B\left(r_k -R_k\bm{\alpha}^{(k)}\right)\label{eq:AA-rk1-alt}
\end{align}
Thanks to \eqref{eq:AA-lsp-Pre} and \eqref{eq:AA-rk1-alt}, we obtain $B^{-1}r_{k+1}=r_k-R_k\bm{\alpha}^{(k)}$, which gives \eqref{eq:AA-LSP-inverseG-P}.  
From the property of the least-squares problem, we have
\begin{equation}
    (B^{-1}r_{k+1})^T R_k=0.
\end{equation}
According to the definition of $R_k$, we rewrite the above orthogonal property as $(B^{-1}r_{k+1})^T(r_{k-s}-r_k)=0$ for $s=1,2,\cdots, m_k$. It follows that  $(B^{-1}r_{k+1})^T(r_{k-i}-r_{k-j})=0$ for all $i,j =1,2,\cdots, m_k$. 

Recall that $B^{-1}r_{k+1}=r_k-R_k\bm{\alpha}^{(k)}$. It leads to $$(B^{-1}r_{k+1})^T(B^{-1}r_{k+1})=(B^{-1}r_{k+1})^T(r_k-R_k\bm{\alpha}^{(k)})=(B^{-1}r_{k+1})^Tr_k.$$ Thus, $\|B^{-1}r_{k+1}\|\leq \|r_k\|$. Using $\|B\|<1$, we have
$$\|r_{k+1}\|=\|B B^{-1}r_{k+1}\|\leq \|B\| \|B^{-1}r_{k+1}\| \leq \|B\|\|r_k\|< \|r_k\|,$$ which is the desired result.
\end{proof} 
We mention that in Theorem \ref{thm:AA-org-P} $m_k$ can be $k$, i.e., it is the full AA method.
\begin{remark}
In \cite[Corollary 2.10]{walker2011anderson}, it discusses the full AA applied to \eqref{eq:PR-FP} and GMRES  applied to $PAx=Pb$, and it shows that if we assume that the sequence $\{\|r_k^G=P(Ax_k^G-b)\|\}$ is strictly decreasing, then $x_{k+1}=q(x_k^G)=x_k^G-r_k^G$. In this situation, based on our  discussion above, we can show that the sequence $\{\|B^{-1}r_k\|\}$ of AA is strictly decreasing.
\end{remark}

\subsection{NGMRES on preconditioned Richardson iteration}
 NGMRES was originally proposed by Washio and Oosterlee \cite{washio1997krylov} to accelerate nonlinear multigrid methods. We present the NGMRES framework in Algorithm \ref{alg:NGMRES}. The main differences between AA and NGMRES are the last term in \eqref{eq:xkp1-NG}  and the least-squares problem defined by \eqref{eq:min-NG}. For convenience, let $\bm{\beta}^{(k)}=\big(\beta_0^{(k)},\beta_1^{(k)},\cdots, \beta_{m_k}^{(k)}\big)$ .
  	\begin{algorithm}[H] 
 		\caption{NGMRES($m$):  Nonlinear GMRES with depth $m$} \label{alg:NGMRES}
 		\begin{algorithmic}[1]
 			\STATE  \textbf{input:} $x_0$  and $m\geq0$ with $m\in\mathbb{N}$
 			\FOR {$k=0,1,\cdots$ until convergence}
 			\STATE compute 
 			\begin{equation}\label{eq:xkp1-NG} 
 				x_{k+1} = q(x_k) + \sum_{i=0}^{m_k}\beta_i^{(k)} \left(q(x_k)- x_{k-i} \right),
 			\end{equation}
 			where $m_k=\min\{k,m\}$ and $\bm{\beta}^{(k)}=\big(\beta_0^{(k)},\beta_1^{(k)},\cdots, \beta_{m_k}^{(k)}\big)$ is obtained by solving the  least-squares problem
 			\begin{equation}\label{eq:min-NG} 
 				\min_{\bm{\beta}^{(k)}} \left\| g(q(x_k))+\sum_{i=0}^{m_k} \beta_i^{(k)} \left(g(q(x_k))-g(x_{k-i}) \right) \right\|^2.
 			\end{equation}
 			\ENDFOR
            \STATE  \textbf{output:} $x_{k+1}$
 		\end{algorithmic}	 
 	\end{algorithm}
 We provide the derivation of NGMRES Algorithm \ref{alg:NGMRES} below. Consider solving $g(x)=0$. Given an iterative scheme $q(x)$, we aim to accelerate the sequence of iterates, where the current iterate is given by $x^q_k=q(x_k)$. The new update is based on the linear approximation of the nonlinear operator $g$ around $x_k^q$ in the space 
 \begin{equation*}
    x^q_k+ {\rm span} \{x^q_k-x_{k-m}, x^q_k-x_{k-{m+1}},\cdots, x^q_k-x_k\}.
 \end{equation*}
We approximate $g(x)$ near $x_k^q$ by
 \begin{align}
     g(x_k^q+ \sum_{i=0}^{m}\beta_i (x^q_k-x_{k-i}))&\approx g(x_k^q) +\sum_{i=0}^m\beta_i \left(\frac{\partial g}{\partial x}\right)_{x_k^q}( x^q_k-x_{k-i}) \label{eq:Taylor-exp-der}  \\
     &\approx g(x_k^q) +\sum_{i=0}^m\beta_i (g(x_k^q)- g(x_{k-i})).\label{eq:Taylor-exp}
 \end{align}
We choose the parameters $\{\beta_i\}_{i=0}^{m}$ that minimize the $2$-norm of the right-hand side in \eqref{eq:Taylor-exp}. Using $\{\beta_i\}_{i=0}^{m}$, we define the new update $x_{k+1}$ as 
\begin{equation}
    x_{k+1}= x_k^q+ \sum_{i=0}^{m}\beta_i (x_k^q-x_{k-i}).
\end{equation}
At each iteration, we update the combination coefficients $\{\beta_i\}_{i=0}^{m}$. This leads to the NGMRES method presented in Algorithm \ref{alg:NGMRES}.  In practice, the evaluation of function $g(\cdot)$ in \eqref{eq:min-NG}  is relatively inexpensive compared to the evaluation of function $q(\cdot)$. Consequently, the computational cost of NGMRES is comparable to that of the underlying fixed‑point iteration.

In \cite{GreifHe25NGMRES}, we have shown that when applying NGMRES to the unpreconditioned Richardson iteration to solve $g(x)=Ax-b=0$, the 2-norm of the classical residual of NGMRES is nonincreasing and the least-squares problem defined in \eqref{eq:min-NG} minimizes the classical $(k+1)$th residual (i.e., $\bar{r}_{k+1}=Ax_{k+1}-b$) of NGMRES. Note that the Taylor approximations in \eqref{eq:Taylor-exp-der} and \eqref{eq:Taylor-exp} become equalities because $q'(x)=A$, a constant matrix. For the preconditioned Richardson iteration, $q'(x)=PA$. It can be shown that NGMRES minimizes the classical residual for the preconditioned Richardson iteration. We state this in the following theorem.

\begin{theorem}
Consider NGMRES($m$) presented in Algorithm \ref{alg:NGMRES} to accelerate preconditioned Richardson iteration, \eqref{eq:PR-FP}. Let $H=I-AP$ and $\bar{r}_{k+1}=Ax_{k+1}-b$. Then, the least-squares problem defined by \eqref{eq:min-NG}  can be rewritten as 
\begin{equation}\label{eq:NG-LSP-P}
 \min_{{\bm \beta}^{(k)}} \|\bar{r}_{k+1}\|^2= \min_{{\bm \beta}^{(k)}} \|H\bar{r}_k-F_k{\bm \beta}^{(k)}\|^2,
\end{equation}
where
\begin{equation}\label{eq:defFk}
    F_k=[\bar{r}_k -H\bar{r}_k, \bar{r}_{k-1}-H\bar{r}_k, \cdots, \bar{r}_{k-m_k}-H\bar{r}_k].
\end{equation}
Moreover, for $i, j=0,1,2, \dots, m_k$, the classical residuals satisfy
\begin{equation*}
\bar{r}_{k+1}^TAP\bar{r}_k=0, \quad  \bar{r}_{k+1}^T (\bar{r}_{k-j} - \bar{r}_{k-i})=0,
\end{equation*}
and
\begin{equation}\label{eq:NG-decrease-P}
 \|\bar{r}_{k+1}\|\leq  \|\bar{r}_k\|.
\end{equation}
\end{theorem}
 \begin{proof}
In Algorithm \ref{alg:NGMRES}, using $q(x_k)=(I-PA)x_k+Pb$, we rewrite \eqref{eq:xkp1-NG} as
\begin{equation}\label{eq:xkp1-NG-expr}
    x_{k+1} = (I-PA)x_k+Pb + \sum_{i=0}^{m_k}\beta_i^{(k)} ((I-PA)x_k+Pb  -x_{k-i}).  
\end{equation}
Using $g(x)=Ax-b$, we have
\begin{align*}
    & g(q(x_k))+\sum_{i=0}^{m_k} \beta_i^{(k)} \left(g(q(x_k))-g(x_{k-i}) \right)\\
    =& A((I-PA)x_k+Pb)-b +  \sum_{i=0}^{m_k} \beta_i^{(k)} \left( A((I-PA)x_k+Pb)-b - (Ax_{k-i} -b) \right)\\
    =& A \left( (I-PA)x_k+Pb +   \sum_{i=0}^{m_k} (  (I-PA)x_k+Pb  -x_{k-i})  \right) -b\\
   = & Ax_{k+1}-b\\
   =&\bar{r}(x_{k+1}).
\end{align*}  
Thus, we obtain the left term in \eqref{eq:NG-LSP-P}.

Next, we rewrite $\bar{r}(x_{k+1})$. From the above derivation, we have
\begin{align*}
  \bar{r}(x_{k+1}) =& A \big( (I-PA)x_k+Pb +   \sum_{i=0}^{m_k} (  (I-PA)x_k+Pb  -x_{k-i})  \big) -b\\
  =& (I-AP)Ax_k-(I-AP)b +\sum_{i=0}^{m_k} \beta_i^{(k)} \left( A(I-PA)x_k+APb -Ax_{k-i}) \right)\\
  =&(I-AP)(Ax_k-b) + \sum_{i=0}^{m_k} \beta_i^{(k)}\big( (I-AP)(Ax_k-b)-(Ax_{k-i}-b)   \big)\\
  =& H\bar{r}_k + \sum_{i=0}^{m_k} \beta_i^{(k)} (H \bar{r}_k -\bar{r}_{k-i}) \\
  =& H\bar{r}_k -F_k {\bm \beta}^{(k)},
\end{align*}
which is the right term in \eqref{eq:NG-LSP-P}. Since ${\bm \beta^{(k)}}$ is the solution of the least-squares problem defined by \eqref{eq:NG-LSP-P}, we have
\begin{equation*}
 \bar{r}_{k+1}^T F_k=0,
\end{equation*}
which means $\bar{r}_{k+1}$ is orthogonal to all columns of $F_k$ defined in \eqref{eq:defFk}. Using $H=I-AP$, we have $\bar{r}_{k+1}^T AP\bar{r}_k=\bar{r}_{k+1}^T (\bar{r}_k -H\bar{r}_k)=0$.
 
For $s=0, 1, 2,\cdots, m_k$, 
\begin{equation*}
     \bar{r}_{k+1}^T (\bar{r}_{k-s}-\bar{r}_k) =\bar{r}_{k+1}^T(\bar{r}_{k-s}-\bar{r}_k+AP\bar{r}_k)=\bar{r}_{k+1}^T(\bar{r}_{k-s}-H\bar{r}_k)=0.
\end{equation*}
It follows that $\bar{r}_{k+1}^T (\bar{r}_{k-j}- \bar{r}_{k-i})=\bar{r}_{k+1}^T (\bar{r}_{k-j}-\bar{r}_k- (\bar{r}_{k-i}-\bar{r}_k))=0$. 
 
Using $\bar{r}_{k+1}^T AP\bar{r}_k=0$ and $\bar{r}_{k+1}^TF_k=0$, we have
 \begin{equation*}
     \bar{r}_{k+1}^T \bar{r}_{k+1} =\bar{r}_{k+1}^T(H\bar{r}_k-F_k{\bm \beta}^{(k)})=\bar{r}_{k+1}^T(I-AP)\bar{r}_k=\bar{r}_{k+1}^T\bar{r}_k.
 \end{equation*}
 It follows $\|\bar{r}_{k+1}\|\leq \|\bar{r}_k\|$.
 \end{proof}
The above theorem reveals that at each iteration,  NGMRES minimizes the classical residual $\bar{r}_k=Ax_k-b$ and the 2-norm of the classical residual is nonincreasing.   Recall that for right-preconditioned GMRES, it minimizes the classical residual too. This makes us wonder if these two approaches are equivalent. For $P=I$,  GMRES and  NGMRES are closely related. We summarize some key results from \cite{GreifHe25NGMRES}.  Assume that the 2-norm of the classical residuals of GMRES is strictly decreasing, then  
   \begin{itemize}
       \item  The iterates of full NGMRES  are the same as those of GMRES. 
       \item  For $A$ that are symmetric or shifted skew-symmetric of the form $A=\tau I + S$, where $\tau \in \mathbb{R}$ and $S$ is skew-symmetric,  the windowed NGMRES($m$) and GMRES are equivalent for $\forall m\in\mathbb{Z}^+$.
   \end{itemize}
We extend the above results to full NGMRES for the preconditioned case. 
 \begin{theorem}\label{thm:fullNG=G}
Consider the full NGMRES presented in Algorithm \ref{alg:NGMRES}, used to accelerate the preconditioned Richardson iteration \eqref{eq:PR-FP} for solving $g(x)=Ax-b=0$. Also consider the right-preconditioned GMRES with the right preconditioner $P$ for solving $Ax=b$.  Assume that both methods use the same initial guess $x_0$. Let $\{x_i\}$ and $\{x_i^{Gr}\}$ be the sequences generated by the full NGMRES and the right-preconditioned GMRES, respectively, and $\bar{r}^{Gr}_i=Ax^{Gr}_i-b$. Furthermore, assume for some $k^*\in \mathbb{Z}^+$, $\bar{r}^{Gr}_{k^*-1}\neq 0$  and $\|\bar{r}^{Gr}_k\|< \|\bar{r}^{Gr}_{k-1}\|$  for all $k\in\mathbb{Z}^+$ such that $0<k<k^*$. Then, the full NGMRES and the right-preconditioned GMRES are equivalent, i.e., $x_k=x_k^{Gr}$ for $0\leq k\leq k^*$.
\end{theorem}
This theorem indicates that the full NGMRES can be an alternative to the right-preconditioned GMRES, which is an extension of the unpreconditioned result in \cite{GreifHe25NGMRES}.  Before presenting the proof, we outline its main idea. 
Recall that for the right-preconditioned GMRES, the $k$th approximation $x_k^{Gr}$ can be expressed as (see \cite[Chapter 9]{saad2003iterative}) 
\begin{equation}\label{eq:RG-xk}
    x_k^{Gr}=x_0+Ps_{k-1}(AP)\bar{r}_0,
\end{equation}
where $\bar{r}_0=Ax_0-b$ and $s_{k-1}$ is a polynomial of degree $k-1$, which minimizes the $k$th residual norm, i.e.,
\begin{equation}\label{eq:right-pre-GMRES-mini}
    \min_{s_{k-1} \in S_{k-1}}\|\bar{r}^{Gr}_k\|=\min_{s_{k-1} \in S_{k-1}}\|\bar{r}_0+APs_{k-1}(AP)\bar{r}_0\|,
\end{equation}
where $S_{k-1}$ denotes the set of polynomials of degree at most $k-1$.
The corresponding right-preconditioned Krylov subspace is 
\begin{equation*}
    \mathcal{K}_k(AP, \bar{r}_0)={\rm span} \{\bar{r}_0, AP\bar{r}_0, \cdots, (AP)^{k-1}\bar{r}_0\}.
\end{equation*}
Then, \eqref{eq:RG-xk} can be rewritten as
\begin{equation}\label{eq:RG-xk-update}
    x_k^{Gr}=x_0+ Pz_k,
\end{equation}
where $z_k\in \mathcal{K}_k(AP, \bar{r}_0)$, and  \eqref{eq:right-pre-GMRES-mini} is
\begin{equation}\label{eq:right-GMRES-mini}
    \min_{z_k\in  \mathcal{K}_k(AP, \bar{r}_0)}\|\bar{r}^{Gr}_k\|= \min_{z_k\in  \mathcal{K}_k(AP, \bar{r}_0)}\|\bar{r}_0+APz_k\|.
\end{equation}
The central idea of the proof of Theorem \ref{thm:fullNG=G} is to show  that the Krylov subspace $\mathcal{K}_k(AP, \bar{r}_0)$ coincides with the span of $V_k$ defined below.
Let $\{x_j\}$ be the sequence generated by the full NGMRES. Given $k\in \mathbb{N}$, where $1\leq k\leq k^*$, we define $ V_k=\{ y_j\}_{j=1}^k$, where 
\begin{equation} \label{eq:defyj}
 y_j = \left\{
\begin{array}{ll}
P^{-1}(x_j-x_0)& \text{if } j=1,2,\cdots, k-1, \\
P^{-1}(x_{k-1}-x_0)-\bar{r}_{k-1} & \text{if } j=k.
\end{array}
\right.
\end{equation}
 
\begin{proof}
To prove $x_k=x_k^{Gr}$ for $1\leq k\leq k^*$, we only need to prove the following two Claims:
\begin{enumerate} 
    \item[(a)] For $1\leq k\leq k^*$, if $V_k$ is a basis for $\mathcal{K}_k(AP,r_0)$, then $x_k=x_k^{Gr}$.
    \item[(b)]  For $1\leq k\leq k^*$,  $V_k$ is a basis for $\mathcal{K}_k(AP,r_0)$.
\end{enumerate}
We first prove Claim (a). For a given $k$, where  $1\leq k\leq k^*$, from \eqref{eq:defyj}, we have  
\begin{equation} \label{eq:defPyj}
 Py_j = \left\{
\begin{array}{ll}
x_j-x_0& \text{if } j=1,2,\cdots, k-1,\\
x_{k-1}-x_0-P\bar{r}_{k-1} & \text{if } j=k.
\end{array}
\right.
\end{equation}
From \eqref{eq:xkp1-NG-expr} and \eqref{eq:defPyj}, the  $k$th iterate of the full NGMRES is given by
 \begin{align*}
x_k&=x_{k-1}-P\bar{r}_{k-1}+ \sum_{i=0}^{k-1} \beta_i^{(k-1)}  (x_{k-1}-P\bar{r}_{k-1}-x_{k-1-i})\\
 &=x_{k-1}-P\bar{r}_{k-1}+ \sum_{i=0}^{k-1} \beta_i^{(k-1)}  ((x_{k-1}-P\bar{r}_{k-1}-x_0)-(x_{k-1-i}-x_0))\\
&=x_0+Py_k + \sum_{i=0}^{k-2} \beta_i^{(k-1)} (Py_k-Py_{k-1-i})+\beta_{k-1}^{(k-1)}Py_k\\
&=x_0+ (1+\sum_{i=0}^{k-1} \beta_i^{(k-1)} )Py_k - \sum_{i=0}^{k-2} \beta_i^{(k-1)}Py_{k-1-i}\\
&=x_0-P\sum_{i=1}^{k}\gamma_i y_i,
\end{align*}
where $\gamma_i=\beta_{k-1-i}^{(k-1)}$ for $i=1,2,\cdots, k-1$ and $\gamma_k=-(1+\sum_{i=0}^{k-1} \beta_i^{(k-1)} )$. Let $\bm{\gamma}=(\gamma_1,\gamma_2,\cdots,\gamma_k)^T$. We know that the coefficient $\bm{\beta}^{(k-1)}$ of NGMRES solves the following least-squares problem
\begin{equation*}
    \min_{\bm{\beta}^{(k-1)}}\|Ax_k-b\|= \min_{\bm{\gamma}}\|\bar{r}_0-AP\sum_{i=1}^{k}\gamma_i y_i\|=\min_{v\in {\rm span}(V_k)}\|\bar{r}_0-APv\|. 
\end{equation*}
By the assumption that $V_k$ is a basis for $\mathcal{K}_k(AP,r_0)$  and the property of the right-preconditioned GMRES,  i.e., \eqref{eq:right-GMRES-mini}, we have $x_k=x_k^{Gr}$. 

Next, we first prove Claim (b) proceeding by induction on $k$. For $k=1$, $V_1=\{-\bar{r}_0\}$, which is obviously a basis for $\mathcal{K}_1(AP,\bar{r}_0)={\rm span}\{\bar{r}_0\}$. Thus, suppose that for $k< k^*$,  
\begin{equation*}  
    V_k=\{P^{-1}(x_1-x_0), \cdots, P^{-1}(x_{k-1}-x_0), P^{-1}(x_{k-1}-x_0)-\bar{r}_{k-1}\},
\end{equation*}
is a basis for $\mathcal{K}_k(AP,\bar{r}_0)$,
and we prove that  
 \begin{equation*}
    V_{k+1}=\{P^{-1}(x_1-x_0),  \cdots, P^{-1}(x_{k-1}-x_0), P^{-1}(x_{k}-x_0), P^{-1}(x_{k}-x_0)-\bar{r}_{k}\},
\end{equation*}
is a basis for $\mathcal{K}_{k+1}(AP,\bar{r}_0)$. Since $V_k$ is a basis for $\mathcal{K}_k(AP,\bar{r}_0)$, from Claim (a), we have $x_k=x_k^{Gr}=x_0+Pz$, where $z\in \mathcal{K}_k(AP,\bar{r}_0)$. Then $\bar{r}_k=\bar{r}_k^{Gr}=\bar{r}_0+APz \in \mathcal{K}_{k+1}(AP,\bar{r}_0)$.  Next, we show the following two properties:
\begin{itemize}
    \item $\bar{r}_k\notin \mathcal{K}_k(AP,\bar{r}_0)$;
    \item $z=P^{-1}(x_{k}-x_0) \notin \mathcal{K}_{k-1}(AP,\bar{r}_0)$.
\end{itemize}
We establish the above results using a proof by contradiction. If $\bar{r}_k\in \mathcal{K}_k(AP,\bar{r}_0)$, then $z\in \mathcal{K}_{k-1}(AP,\bar{r}_0)$. We can define $\tilde{x}=x_0+Pz$, which satisfies $\|A\tilde{x}-b\|<\|\bar{r}_{k-1}\|$ contradicting the fact that $x^{Gr}_{k-1}$ is the solution at the $(k-1)$th step of GMRES that minimizes the residual. This proves $\bar{r}_k\notin \mathcal{K}_k(AP,\bar{r}_0)$.

If $z=P^{-1}(x_k-x_0) \in \mathcal{K}_{k-1}(AP,\bar{r}_0)$, then $APz=\bar{r}_k-\bar{r}_0\in\mathcal{K}_k(AP,\bar{r}_0)$. This implies $\bar{r}_k \in\mathcal{K}_k(AP,\bar{r}_0)$, which contradicts the previous conclusion $\bar{r}_k\notin \mathcal{K}_k(AP,\bar{r}_0)$.

Note that the first $(k-1)$ elements of $V_k$ and $V_{k+1}$ are identical, linearly independent, and belong to $\mathcal{K}_{k-1}(AP,\bar{r}_0)$. Moreover, $P^{-1}(x_{k}-x_0)\in\mathcal{K}_k(AP,\bar{r}_0)$ but $P^{-1}(x_{k}-x_0)\notin\mathcal{K}_{k-1}(AP,\bar{r}_0)$. Finally, $P^{-1}(x_{k}-x_0)-\bar{r}_{k}\in\mathcal{K}_{k+1}(AP,\bar{r}_0)$ while $P^{-1}(x_{k}-x_0)-\bar{r}_{k}\notin\mathcal{K}_{k}(AP,\bar{r}_0)$.  This means that $V_{k+1}$ is a basis for $\mathcal{K}_{k+1}(AP,\bar{r}_0)$.
\end{proof}

\begin{example}\label{ex:Laplace}
We validate the above theorem by considering the two-dimensional Laplace equation. 
We construct $A$ using a centered‑difference discretization on an $N\times N$ grid with $h=1/(N+1)$,
scaling the matrix by $h^2$. The right‑hand side $b$ is taken to be a vector of ones to form a test problem. We use $P=L^{-1}$ as the preconditioner, where $L$ denotes the lower triangular part of $A$, and use the zero vector as the initial guess.  
\end{example}
\begin{figure}[H]
		\centering
		\includegraphics[width=0.49\textwidth]{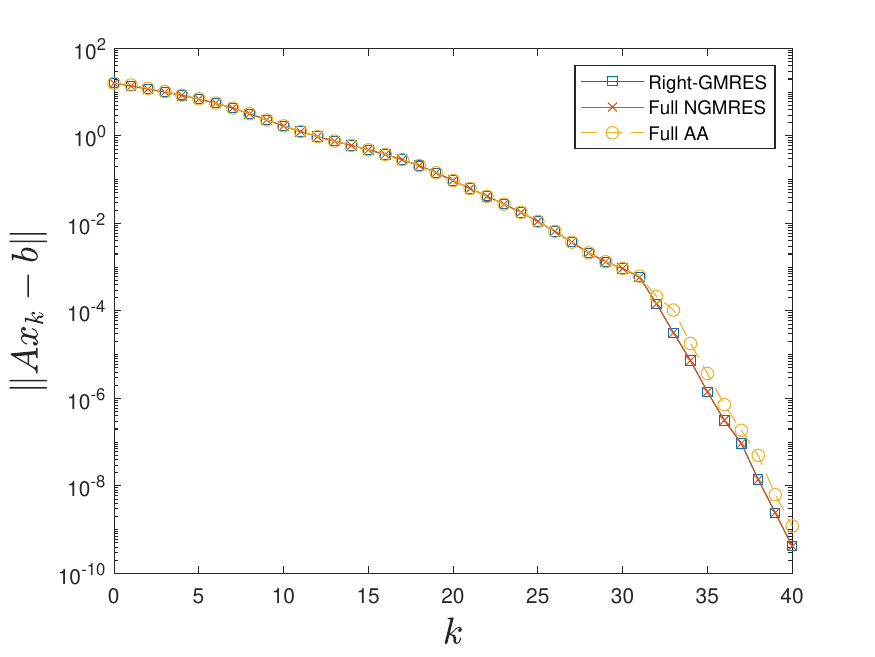}
        \includegraphics[width=0.49\textwidth]{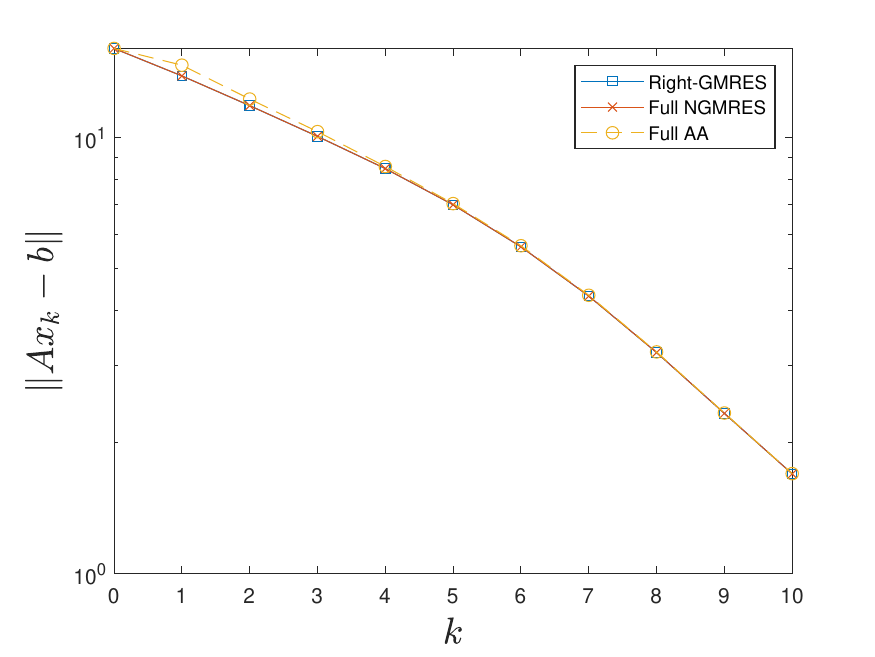}
		\caption{Example \ref{ex:Laplace} with $N=16$. Convergence history (residual norms) for the right-preconditioned GMRES, full NGRMES and full AA. The graph on the right shows the first 10 of the 40 iterations.}\label{fig:RGM-NG-AA}
\end{figure}
In the left-hand graph of Figure \ref{fig:RGM-NG-AA}, we see that the residual norm of right-preconditioned GMRES matches that of full NGMRES. In contrast, the residual norm of full AA is slightly larger than that of right-preconditioned GMRES (see the right‑hand graph). These results are consistent with our theorem. We remark that choosing small $N=16$ in the test ensures that the computing of $L^{-1}$ is sufficiently accurate.

In practice, we prefer to use a finite value of $m$ in NGMRES method. Although there is no general theoretical guidance on choosing $m$, \cite{GreifHe25NGMRES} shows that the depth $m$ does not affect the performance of NGMRES($m$) when applied to the unpreconditioned Richardson iteration for two special classes of matrices $A$. We extend this result to the preconditioned case.

 \begin{theorem}\label{thm:NGMRESm=NGMRES1-right}
Consider NGMRES($m$) presented in Algorithm \ref{alg:NGMRES}, used to accelerate preconditioned Richardson iteration  \eqref{eq:PR-FP} for solving $g(x)=Ax-b=0$, and also the right-preconditioned GMRES with the right preconditioner $P$ for solving $Ax=b$.  Assume $(AP)^T=AP$ or $AP=\tau I + S$, where $\tau \in \mathbb{R}$ and  $S^T=-S$, and both methods use the same initial guess $x_0$. Let $\{x_i^{NG(m)}\}$ and $\{x_i^{Gr}\}$ be the sequences generated by  NGMRES($m$) and the right-preconditioned GMRES, respectively. Furthermore, assume for some $k^*\in \mathbb{Z}^+$, $\bar{r}^{Gr}_{k^*-1}\neq 0$ and $\|\bar{r}^{Gr}_k\|< \|\bar{r}^{Gr}_{k-1}\|$  for all $k\in\mathbb{Z}^+$ such that $0<k<k^*$.  Then, $x_k^{NG(m)}=x_k^{Gr}$ for $\forall m\in\mathbb{Z}^+$ and  $0\leq k\leq k^*$.
\end{theorem}

We first present a lemma that will be used for the proof.
\begin{lemma}\label{lem:rk-rightGMRES-orth}
For $0\leq i\leq j\leq k+1$, the residuals of the right-preconditioned GMRES satisfy
\begin{equation}
\bar{r}_{k+1}^T(AP)\bar{r}_k=0,\quad\bar{r}_{k+1}^T(\bar{r}_j-\bar{r}_i)=0, \quad \text{and}\quad  \bar{r}_{k} \perp AP\mathcal{K}_k(AP,\bar{r}_0)
\end{equation}
\end{lemma}
Since the proof of Theorem \ref{thm:NGMRESm=NGMRES1-right} can be obtained directly from the proof of \cite[Theorem 3.6]{GreifHe25NGMRES}. In the following, we present a sketch of the proof of Theorem \ref{thm:NGMRESm=NGMRES1-right} 
\begin{proof}
The key step in the proof of Theorem \ref{thm:NGMRESm=NGMRES1-right} is based on \cite[Lemma 3.7]{GreifHe25NGMRES}, where we replace $A$ with $AP$. We then proceed by following the proof of \cite[Theorem 3.6]{GreifHe25NGMRES}, replacing $A$ with $AP$ and $M=I-A$ with $H=I-AP$, applying the Lemma \ref{lem:rk-rightGMRES-orth}, and finally using induction. 
\end{proof}
Theorem \ref{thm:NGMRESm=NGMRES1-right} shows that for a nonsymmetric $A$, one can construct a matrix $P$ such that $AP$ becomes symmetric. In this case, applying NGMRES($m$) with $m=1$ yields the same performance as NGMRES($m$) with $m>1$. Consequently, this approach reduces the computational cost. This may provide new insight into the development of efficient numerical methods for nonsymmetric linear systems.

 \section{New Variants Using Modified Least-Squares Problem}\label{sec:new}	
From the previous sections, we know that the least-squares problem for NGMRES minimizes the classical residual, whereas the corresponding problem in AA minimizes neither the classical residual nor the preconditioned residual. Motivated by the derivation of NGMRES, we propose two new variants of AA and  one of NGMRES by using modified least-squares problems. Specifically, for AA we consider choosing ${\bm \alpha}^{(k)}$ that minimizes either $g(x_{k+1})$ (i.e., classical residual in the linear case), which corresponds to the new approach presented in Algorithm \ref{alg:AAg}, denoted as AAg($m$),  or minimizes $r_{k+1}$ (i.e., the preconditioned residual in the linear case) given in Algorithm \ref{alg:AAr}, denoted as AAr($m$). For NGMRES, we consider choosing ${\bm \beta}^{(k)}$ that minimizes  $r_{k+1}$, leading to Algorithm \ref{alg:NGMRESr}, denoted as NGMRESr($m$).

In the following, we will analyze the properties of these new approaches applied to the preconditioned Richardson iteration for solving $g(x)=Ax-b=0$.

	\begin{algorithm}[H] 
		\caption{AAg($m$):  Anderson acceleration with g-norm} \label{alg:AAg}
		\begin{algorithmic}[1] 
			\STATE  \textbf{input:} $x_0$  and $m\geq0$ with $m\in\mathbb{N}$
			\FOR {$k=0,1,\cdots$ until convergence }
			\STATE compute 
			\begin{equation*} 
				x_{k+1} = q(x_k) + \sum_{i=1}^{m_k}\alpha_i^{(k)} \left(q(x_k)- q(x_{k-i}) \right),
			\end{equation*}
			where $m_k=\min\{k,m\}$ and $\bm{\alpha}^{(k)}=\big(\alpha_1^{(k)},\cdots, \alpha_{m_k}^{(k)}\big)$ is obtained by solving the  least-squares problem
			\begin{equation}\label{eq:min-AAg}
	\min_{\bm{\alpha}^{(k)}} \left\|g(q(x_k))+\sum_{i=1}^{m_k} \alpha_i^{(k)} \left(g(q(x_k))-g(q(x_{k-i})) \right) \right\|^2.
			\end{equation}
			\ENDFOR
        \STATE  \textbf{output:} $x_{k+1}$
		\end{algorithmic}
	\end{algorithm}

	\begin{algorithm}[H] 
		\caption{AAr($m$):  Anderson acceleration with r-norm} \label{alg:AAr}
		\begin{algorithmic}[1] 
			\STATE  \textbf{input:} $x_0$  and $m\geq0$ with $m\in\mathbb{N}$
			\FOR {$k=0,1,\cdots$ until convergence }
			\STATE compute 
			\begin{equation*} 
				x_{k+1} = q(x_k) + \sum_{i=1}^{m_k}\alpha_i^{(k)} \left(q(x_k)- q(x_{k-i}) \right),
			\end{equation*}
			where $m_k=\min\{k,m\}$ and $\bm{\alpha}^{(k)}=\big(\alpha_1^{(k)},\cdots, \alpha_{m_k}^{(k)}\big)$ is obtained by solving the  least-squares problem
			\begin{equation}\label{eq:min-AAr}
				\min_{\bm{\alpha}^{(k)}} \left\|r(q(x_k))+\sum_{i=1}^{m_k} \alpha_i^{(k)} \left(r(q(x_k))-r(q(x_{k-i})) \right) \right\|^2.
			\end{equation}
			\ENDFOR
            \STATE  \textbf{output:} $x_{k+1}$
		\end{algorithmic}
	\end{algorithm}
 	\begin{algorithm}[H] 
 		\caption{NGMRESr($m$):  Nonlinear GMRES with r-norm} \label{alg:NGMRESr}
 		\begin{algorithmic}[1]
 			\STATE  \textbf{input:} $x_0$  and $m\geq0$ with $m\in\mathbb{N}$
 			\FOR {$k=0,1,\cdots$ until convergence}
 			\STATE compute 
 			\begin{equation*} 
 				x_{k+1} = q(x_k) + \sum_{i=0}^{m_k}\beta_i^{(k)} \left(q(x_k)- x_{k-i} \right),
 			\end{equation*}
 			where $m_k=\min\{k,m\}$ and $\bm{\beta}^{(k)}=\big(\beta_0^{(k)},\beta_1^{(k)},\cdots, \beta_{m_k}^{(k)}\big)$ is obtained by solving the  least-squares problem
 			\begin{equation}\label{eq:min-NGr} 
 				\min_{\bm{\beta}^{(k)}} \left\| r(q(x_k))+\sum_{i=0}^{m_k} \beta_i^{(k)} \left(r(q(x_k))-r(x_{k-i}) \right) \right\|^2.
 			\end{equation}
 			\ENDFOR
        \STATE  \textbf{output:} $x_{k+1}$
 		\end{algorithmic}	 
 	\end{algorithm}
In Algorithm \ref{alg:AAg}, the update $x_{k+1}$ can be rewritten as
\begin{equation*}
    x_{k+1} =(I-PA)x_k+Pb+\sum_{i=1}^{m_k} \alpha_i^{(k)}(I-PA)(x_k-x_{k-i}).
\end{equation*}

\begin{theorem}
Consider AAg($m$) presented in Algorithm \ref{alg:AAg} to accelerate preconditioned Richardson iteration, \eqref{eq:PR-FP}. Let $H=I-AP$ and $\bar{r}_{k+1}=Ax_{k+1}-b$. The least-squares problem defined by \eqref{eq:min-AAg} in  Algorithm \ref{alg:AAg} can be rewritten as
\begin{equation*} 
\min_{ \bm{\alpha}^{(k)}} \|\bar{r}_{k+1}\|^2=\min_{{\bm \alpha}^{(k)}} \|H(\bar{r}_k-\bar{R}_k{\bm \alpha}^{(k)})\|^2,
\end{equation*}
where
\begin{equation*} 
    \bar{R}_k=[\bar{r}_{k-1}-\bar{r}_k, \bar{r}_{k-2}-\bar{r}_k, \cdots, \bar{r}_{k-m_k}-\bar{r}_k].
\end{equation*}
For $i, j=0,1,2, \dots, m_k$, the classical residuals satisfy
\begin{equation*}
  \bar{r}_{k+1}^T H(\bar{r}_{k-j} - \bar{r}_{k-i})=0,
\end{equation*}
and
\begin{equation}\label{eq:AAg-decrease-P}
 \|\bar{r}_{k+1}\|\leq  \|H\bar{r}_k\|.
\end{equation}
Moreover, if $\|H\|<1$,
\begin{equation*}
 \|\bar{r}_{k+1}\|<\|\bar{r}_k\|,
\end{equation*}
which means AAg($m$) will converge. 
\end{theorem}
\begin{proof}
 In Algorithm \ref{alg:AAg}, we simplify
\begin{align*}
   & g(q(x_k))+\sum_{i=1}^{m_k} \alpha_i^{(k)} \left(g(q(x_k))-g(q(x_{k-i})) \right)\\
    =& A((I-PA)x_k+Pb)-b+\sum_{i=1}^{m_k} \alpha_i^{(k)} A(I-PA)\left( x_k-x_{k-i} \right)\\
    =& A\left( (I-PA)x_k +Pb + \sum_{i=1}^{m_k} \alpha_i^{(k)} (I-PA)( x_k-x_{k-i}) \right) - b  \\
     =& A\left( q(x_k) + \sum_{i=1}^{m_k} \alpha_i^{(k)} (q(x_k)- q(x_{k-i}) ) \right) - b  \\
    =&Ax_{k+1}-b\\
    =&g(x_{k+1}).
\end{align*} 
Now, we rewrite the above formula as
\begin{align*}
    \bar{r}_{k+1}=& A((I-PA)x_k+Pb)-b+\sum_{i=1}^{m_k} \alpha_i^{(k)} A(I-PA)\left( x_k-x_{k-i} \right)\\
    =&(I-AP)(Ax_k-b) + \sum_{i=1}^{m_k}\alpha_i^{(k)} (I-AP)(Ax_k-Ax_{k-i})\\
    =&H\bar{r}_k + H\sum_{i=1}^{m_k}\alpha_i^{(k)} (\bar{r}_k-\bar{r}_{k-i})\\
    = &H(\bar{r}_k-\bar{R}_k\bm{\alpha}^{(k)}).
\end{align*}
According to the property of least-squares problem, we have
\begin{equation*}
\bar{r}_{k+1}^TH\bar{R}_k=0.
\end{equation*}
It can be rewritten as
\begin{equation*}
    \bar{r}_{k+1}^TH(\bar{r}_{k-j}-\bar{r}_{k-i})=0.
\end{equation*}
Moreover, from $\bar{r}_{k+1}^T\bar{r}_{k+1}=\bar{r}_{k+1}^TH(\bar{r}_k-\bar{R}_k\bm{\alpha}^{(k)})=\bar{r}_{k+1}^TH\bar{r}_k$, we have $$\|\bar{r}_{k+1}\|\leq \|H\bar{r}_k\|\leq \|H\| \|\bar{r}_k\|,$$
which gives the desired results.
\end{proof}

\begin{theorem}
Consider AAr($m$) presented in Algorithm \ref{alg:AAr} to accelerate preconditioned Richardson iteration, \eqref{eq:PR-FP}. Let $B=I-PA$ and $r_{k+1}=P(Ax_{k+1}-b)$. The least-squares problem defined by \eqref{eq:min-AAr} in  Algorithm \ref{alg:AAr} can be rewritten as
\begin{equation*} 
\min_{ \bm{\alpha}^{(k)}} \|r_{k+1}\|^2=\min_{{\bm \alpha}^{(k)}} \|B(r_k-R_k{\bm \alpha}^{(k)})\|^2,
\end{equation*}
where $R_k$ is defined in \eqref{eq:defRk}.

Moreover, for $i, j=0,1,2, \dots, m_k$, the preconditioned residuals satisfy
\begin{equation*}
r_{k+1}^TB(r_{k-j} - r_{k-i})=0,
\end{equation*}
and
\begin{equation}\label{eq:AAg-decrease-P}
 \|r_{k+1}\|\leq  \|Br_k\|.
\end{equation}
If $\|B\|<1$,
\begin{equation*}
 \|r_{k+1}\|<\|r_k\|,
\end{equation*}
which means AAr($m$) will converge. 
\end{theorem}
\begin{proof}
In Algorithm \ref{alg:AAr}, by standard calculation, it can be shown that
\begin{align*}
    &r(q(x_k))+\sum_{i=1}^{m_k} \alpha_i^{(k)} \left( r(q(x_k)) -r(q(x_{k-i}))\right) \\
    =&P(A((I-PA)x_k+Pb)-b) +  PA(I-PA)\sum_{i=1}^{m_k} \alpha_i^{(k)}(x_k-x_{k-i}) \\
    =&P(Ax_{k+1}-b)\\
    =&r(x_{k+1}).
\end{align*}
Note that $r(q(x_j)) =(I-PA)P(Ax_j-b)=(I-PA)r_j$. Next, we rewrite 
\begin{align*}
   r_{k+1}=&r(q(x_k))+\sum_{i=1}^{m_k} \alpha_i^{(k)} \left( r(q(x_k)) -r(q(x_{k-i}))\right)\\
    =&(I-PA)(r_k + \sum_{i=1}^{m_k} \alpha_i^{(k)}(r_k-r_{k-i}) )\\
    =& B (r_k-R_k\bm{\alpha}^{(k)}).
\end{align*}
Using the property of least-squares problem, we have $r_{k+1}^TBR_k=0$, which can be rewritten as $r_{k+1}^TB(r_{k-i}-r_{k-j})=0$ for $0\leq i, j \leq m_k$. Moreover, $$r_{k+1}^Tr_{k+1}=r_{k+1}^TB(r_k-R_k\bm{\alpha}^{(k)})=r_{k+1}^TBr_k,$$
which gives $\|r_{k+1}\|\leq \|Br_k\|\leq \|B\|\|r_k\|$.
\end{proof}

\begin{theorem}\label{thm:NGMRESr}
Consider NGMRESr($m$) presented in Algorithm \ref{alg:NGMRESr} to accelerate preconditioned Richardson iteration, \eqref{eq:PR-FP}. Let $B=I-PA$ and $r_{k+1}=P(Ax_{k+1}-b)$. Then, the least-squares problem defined by \eqref{eq:min-NGr}  can be rewritten as 
\begin{equation}\label{eq:NGr-LSP-P}
 \min_{{\bm \beta}^{(k)}} \|r_{k+1}\|^2= \min_{{\bm \beta}^{(k)}} \|Br_k-E_k{\bm \beta}^{(k)}\|^2,
\end{equation}
where
\begin{equation}\label{eq:defEk}
    E_k=[r_k -Br_k, r_{k-1}-Br_k, \cdots, r_{k-m_k}-Br_k].
\end{equation}
For $i, j=0,1,2, \dots, m_k$, the classical residuals satisfy
\begin{equation*}
r_{k+1}^T(PA)r_k=0, \quad  r_{k+1}^T (r_{k-j} - r_{k-i})=0,
\end{equation*}
and
\begin{equation}\label{eq:NGr-decrease-P}
 \|r_{k+1}\|\leq  \|r_k\|.
\end{equation}
\end{theorem}
\begin{proof}
In Algorithm \ref{alg:NGMRESr}, we have
\begin{align*}
    & r(q(x_k))+\sum_{i=0}^{m_k} \beta_i^{(k)} \left(r(q(x_k))-r(x_{k-i}) \right)\\
    =& P(A( (I-PA)x_k+Pb )-b )  +\sum_{i=0}^{m_k} \beta_i^{(k)} ( P(A( (I-PA)x_k+Pb )-b )-P(Ax_{k-i} -b ))\\
    =& P (Ax_{k+1}-b)= r(x_{k+1}).
\end{align*}
Then, we can rewrite 
\begin{align*}
     & r(q(x_k))+\sum_{i=0}^{m_k} \beta_i^{(k)} \left(r(q(x_k))-r(x_{k-i}) \right)\\
     =&(I-PA)r_k + \sum_{i=0}^{m_k} \beta_i^{(k)} \left((I-PA)r_k-r(x_{k-i}) \right)\\
     =& Br_k -E_k\bm{\beta}^{(k)}.
\end{align*}
Using the property of least-squares problem, we have $r_{k+1}^TE_k=0$. It follows that $r_{k+1}^T(r_{k-s}-Br_k)=0$. For $s=0$, we have $r_{k+1}^T(PAr_k)=0$. Moreover, $r_{k+1}^T(r_{k-i}-r_{k-j})=0$.

Next, we compute
\begin{equation*}
 r_{k+1}^Tr_{k+1}=r_{k+1}^T (Br_k-E_k\bm{\beta}^{(k)})=r_{k+1}^TBr_k=r_{k+1}(I-PA)r_k=r_{k+1}^Tr_k.    
\end{equation*}
It follows that $\|r_{k+1}\|\leq \|r_k\|$.
\end{proof}
Theorem \ref{thm:NGMRESr} reveals that NGMRESr($m$) minimizes preconditioned residual and the 2-norm of the preconditioned residuals is nonincreasing, which is different than AA.
\begin{theorem}
Consider solving $g(x)=Ax-b=0$. Assume that the full NGMRESr and the left-preconditioned GMRES use the same initial guess $x_0$. Let $\{x_i\}$ and $\{x_i^{Gl}\}$ be the sequences generated by the full NGMRESr and the left-preconditioned GMRES, respectively, and $r^{Gl}_i=P(Ax^{Gl}-b)$. Assume that for some $k^*\in \mathbb{Z}^+$, $r^{Gl}_{k^*-1}\neq 0$  and $\|r^{Gl}_k\|< \|r^{Gl}_{k-1}\|$  for all $k$ such that $0<k<k^*$. Then, the full NGMRESr  is equivalent to the left-preconditioned GMRES, i.e., $x_k=x_k^{Gl}$ for $0\leq k\leq k^*$.
\end{theorem}
Before providing the proof, we review the left-preconditioned GMRES. Recall that the $k$th approximation $x_k^{Gl}$ of the left-preconditioned GMRES can be expressed as (see \cite[Chapter 9]{saad2003iterative}) 
\begin{equation}\label{eq:LG-xk}
    x_k^{Gl}=x_0+s_{k-1}(PA)r_0,
\end{equation}
where $r_0=P(Ax_0-b)$ and $s_{k-1}$ is a polynomial of degree $k-1$, which minimizes the $k$th preconditioned residual norm, i.e.,
\begin{equation}\label{eq:left-pre-GMRES-mini}
    \min_{s_{k-1} \in S_{k-1}}\|r^{Gl}_k\|=\min_{s_{k-1} \in S_{k-1}}\|r_0+PAs_{k-1}(PA)r_0\|,
\end{equation}
where $S_{k-1}$ denotes the set of polynomials of degree at most $k-1$.
The corresponding left-preconditioned Krylov subspace is 
\begin{equation*}
    \mathcal{K}_k(PA, r_0)={\rm span} \{r_0, PAr_0, \cdots, (PA)^{k-1}r_0\}.
\end{equation*}
Then, \eqref{eq:LG-xk} can be rewritten as
\begin{equation}\label{eq:LG-xk-update}
    x_k^{Gl}=x_0+ z_k,
\end{equation}
where $z_k\in \mathcal{K}_k(PA, r_0)$, and \eqref{eq:left-pre-GMRES-mini} is given by 
\begin{equation}\label{eq:left-GMRES-mini}
    \min_{z_k\in  \mathcal{K}_k(PA, r_0)}\|r^{Gl}_k\|= \min_{z_k\in  \mathcal{K}_k(PA, r_0)}\|r_0+PAz_k\|.
\end{equation}

\begin{proof}
In \cite{GreifHe25NGMRES}, we have shown that applying Algorithm \ref{alg:NGMRES} to solve $PAx=Pb$, i.e., $g(x)=r(x)=PAx-Pb$, with infinite window size and $q(x)=x+(PAx-Pb)$ yields an iteration process equivalent to that of GMRES applied to $PAx=Pb$. However, GMRES applied to $PAx=Pb$ is the same as the left-preconditioned GMRES with preconditioner $P$ applied to solve $Ax=b$.   Note that  Algorithm \ref{alg:NGMRESr} applying to $q(x)=x+P(Ax-b)$ to solve $Ax=b$ is the same as Algorithm  \ref{alg:NGMRES} applying to $q(x)=x+(PAx-Pb)$ to solve $PAx=Pb$, i.e, $g(x)=PAx-Pb$. Thus, the iterates generated by Algorithm \ref{alg:NGMRESr} is identical to those of the left-preconditioned GMRES. 
\end{proof}

We validate the above theorem by considering Example \ref{ex:Laplace} using $N=16$. In the left-hand graph of Figure \ref{fig:LGM-NGr-NG-AA}, we see that the residual norm of left-preconditioned GMRES matches that of full NGMRESr, which confirms our theorem. In contrast, the residual norms of full AA and full NGMRES are slightly larger than that of left-preconditioned GMRES (see the right‑hand graph). 
    
\begin{figure}[H]
		\centering
		\includegraphics[width=0.49\textwidth]{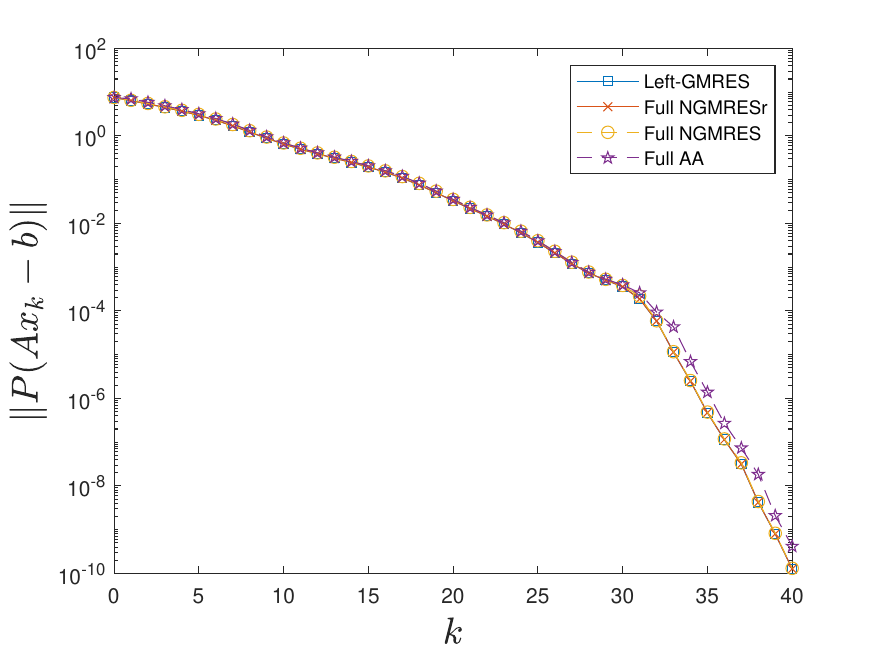}
        \includegraphics[width=0.49\textwidth]{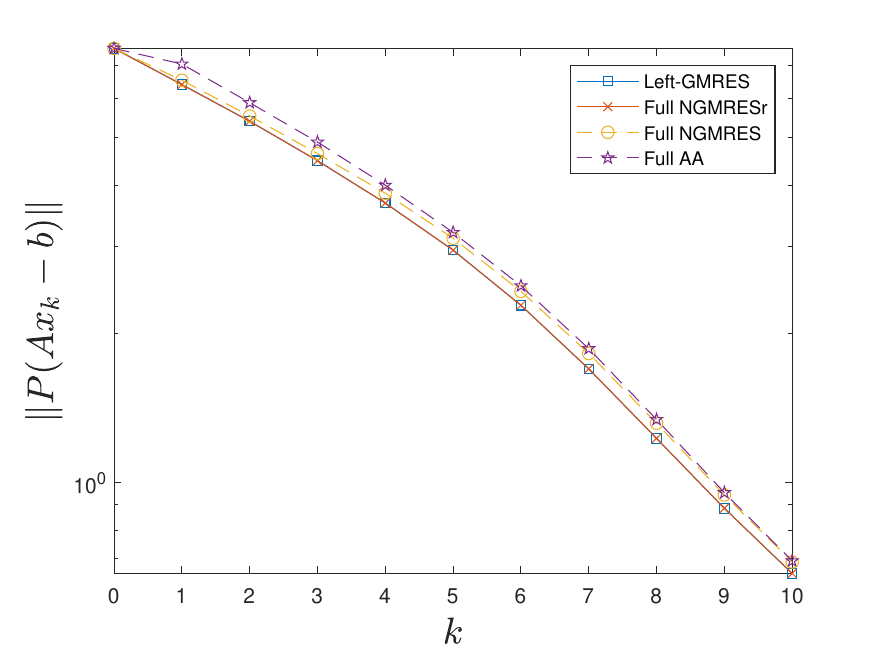}
		\caption{Example \ref{ex:Laplace} with $N=16$. Convergence history (preconditioned residual norms) for left-preconditioned GMRES, full NGRMESr, full NGMRES, and full AA. The graph on the right shows the first 10 of the 40 iterations.}\label{fig:LGM-NGr-NG-AA}
	\end{figure}

In Figure \ref{fig:All-methods}, we compare all methods discussed in the previous sections for Example \ref{ex:Laplace}. The residual norm of full NGMRESr shows the best performance overall. In the early iteration stage, full AAg performs the worst among all methods (see the right‑hand graph). Ultimately, all methods reach the same stopping criterion at around 40 iterations.
\begin{figure}[H]
		\centering
		\includegraphics[width=0.49\textwidth]{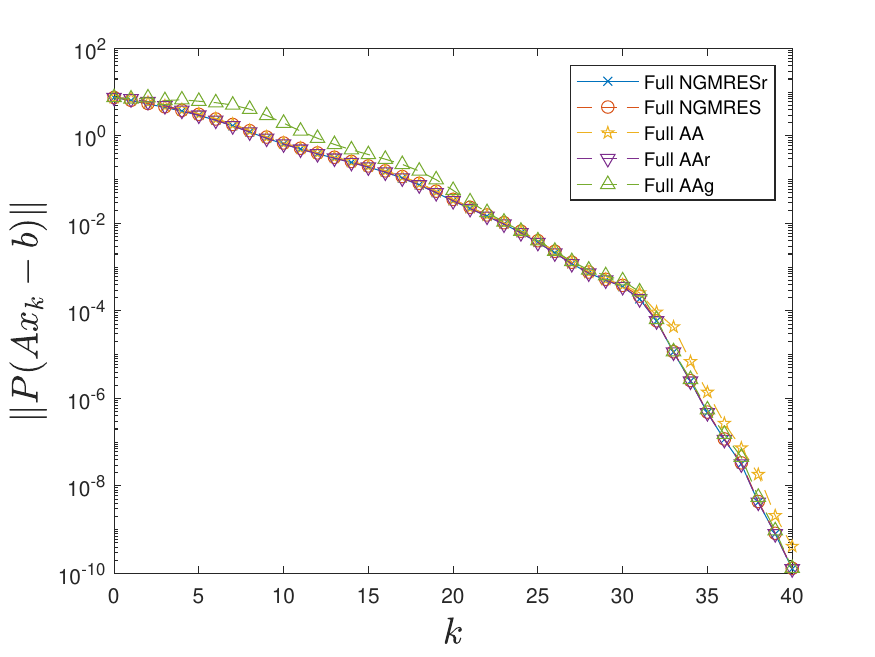}
        \includegraphics[width=0.49\textwidth]{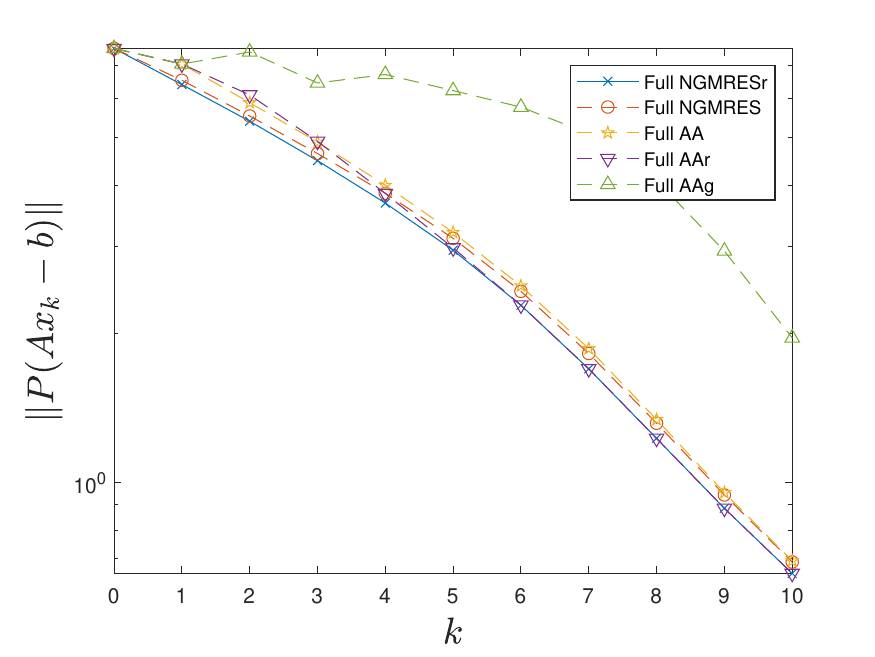}
		\caption{Example \ref{ex:Laplace} with $N=16$. Convergence history (preconditioned residual norms) for full NGRMESr, full NGMRES, full AA, full AAr and full AAg. The graph on the right shows the first 10 of the 40 iterations.}\label{fig:All-methods}
	\end{figure}

 \begin{theorem}
Consider NGMRESr($m$) presented in Algorithm \ref{alg:NGMRESr} to accelerate preconditioned Richardson iteration, \eqref{eq:PR-FP}, and also the left-preconditioned GMRES with left preconditioner $P$ to solve $g(x)=Ax-b=0$. Let $\{r^{Gl}_k\}$ be the sequence generated by the left-preconditioned GMRES. Assume that both methods use the same initial guess, and  for some $k^*\in \mathbb{Z}^+$, $r^{Gl}_{k^*-1}\neq 0$  and $\|r^{Gl}_k\|< \|r^{Gl}_{k-1}\|$  for all $k\in\mathbb{Z}^+$ such that $0<k<k^*$.  If $(PA)^T=PA$ or $PA=\tau I + S$, where $\tau \in \mathbb{R}$ and $S^T=-S$,  we have $x_k^{NG(m)}=x_k^{Gl}$ for $\forall m\in \mathbb{Z}^+$ and $0\leq k\leq k^*$.
\end{theorem}
\begin{proof}
As shown earlier, Algorithm \ref{alg:NGMRESr} applied to $q(x)=x+P(Ax-b)$ for solving $Ax=b$  is equivalent to Algorithm  \ref{alg:NGMRES} applied to $q(x)=x+(PAx-Pb)$ for solving $PAx=Pb$, i.e, $g(x)=PAx-Pb$. We then treat $PA$ as $A$ in \cite[Theorem 3.6]{GreifHe25NGMRES}, which gives the desired results.
\end{proof}
We summarize the five methods in Table \ref{tab:AA-NGMRES-LSP}.  We have analyzed the connections between AA, NGMRES, and NGMRESr and preconditioned GMRES. However, no corresponding relationship between AAg or AAr and GMRES has been established.

\begin{table}[H]
\footnotesize
\centering
\caption{Full AA and full NGMRES based methods applied to accelerate $q(x_k)=x_k+P(b-Ax_k)$ to solve $Ax=b$. Recall $\bar{r}_{k+1}=Ax_{k+1}-b$, $r_{k+1}=P(Ax_{k+1}-b)$, and $B=I-PA$. Let $H=I-AP$, $x_k^{Gl}$ and $x_k^{Gr}$ are the iterates obtained from left- and right-preconditioned  GMRES, respectively. The notation `-' indicates that the relationship is unknown and $\downarrow$ stands for nonincreasing.}\label{tab:AA-NGMRES-LSP}
\begin{tabular}{c|c|l|c}
\hline
Method   & LSP minimizes  & $\downarrow$   & relation with GMRES \\
\hline
AA       & $\|B^{-1}r_{k+1}\|$  &-  & $x_{k+1}=q(x_k^{Gl})$ \\
\hline
AAg       &  $\|\bar{r}_{k+1}\|$  &$\downarrow$ provided $\|H\|<1$    & - \\
\hline
AAr    & $\|r_{k+1}\|$  & $\downarrow$ provided $\|B\|<1$ & - \\
\hline
NGMRES    & $\|\bar{r}_{k+1}\|$  & $\downarrow$   & $x_{k+1}=x_k^{Gr}$\\
\hline
NGMRESr    & $\|r_{k+1}\|$ & $\downarrow$  & $x_{k+1}=x_k^{Gl}$\\
\hline
\end{tabular}
\end{table}

\begin{remark}
In Algorithm \ref{alg:AAg}, the additional evaluation of function $g(\cdot)$ in \eqref{eq:min-AAg} is relative inexpensive compared to the evaluation of $q$. Thus, the computational time of one step in  Algorithm \ref{alg:AAg} is competitive with that of Algorithm \ref{alg:AA}.  In Algorithm \ref{alg:AAr}, we need to evaluate $r(q(\cdot))=q(\cdot)-q(q(\cdot))$ in \eqref{eq:min-AAr}. The dominantly additional cost is the evaluation of $q(q(\cdot))$, which might increase the total computational cost compared to the original AA presented in Algorithm \ref{alg:AA}. Similarly, in Algorithm \ref{alg:NGMRESr}, the dominantly additional cost is the evaluation of $q(q(\cdot))$ in \eqref{eq:min-NGr}, which might be expensive compared to the original NGMRES presented in Algorithm \ref{alg:NGMRES}. As a result, practitioners may need to weigh the trade-off between iteration count and CPU time and efficient implementation of these methods is required. In this work, we do not further expand the discussion on the CPU time.
\end{remark}

\section{Numerical Results}\label{sec:num}
In this section, we use two linear problems to illustrate the performance of AA, NGMRES and their variants based on the reformulated least-squares problems. 
 In practice, windowed variants of AA and NGMRES are commonly used. Accordingly, we consider windowed variants for two linear systems with symmetric and nonsymmetric coefficient matrices. 
 
\subsection{Symmetric case}\label{sec:symm} 
We compare the AA‑type and NGMRES‑type methods with $m=5$ and $m=15$ for Example \ref{ex:Laplace} with $N=64$.  We use $P=L^{-1}$ as the preconditioner, where $L$ denotes the lower triangular part of $A$.
 Figure \ref{fig:m5m15-all-methods} shows the convergence history of the preconditioned residual norms. For $m=5$, the AA‑type methods outperform the NGMRES‑type methods, with AAr(5) delivering the best performance among all methods considered. When $m=15$,  the windowed AA variants continue to outperform the windowed NGMRES‑type methods, and AAr(15) is slightly better than the others. Increasing the window size improves the performance of the NGMRES‑type methods, and NGMRESr and NGMRES behave similarly.

    \begin{figure}[H]
		\centering
        \includegraphics[width=0.3\textwidth]{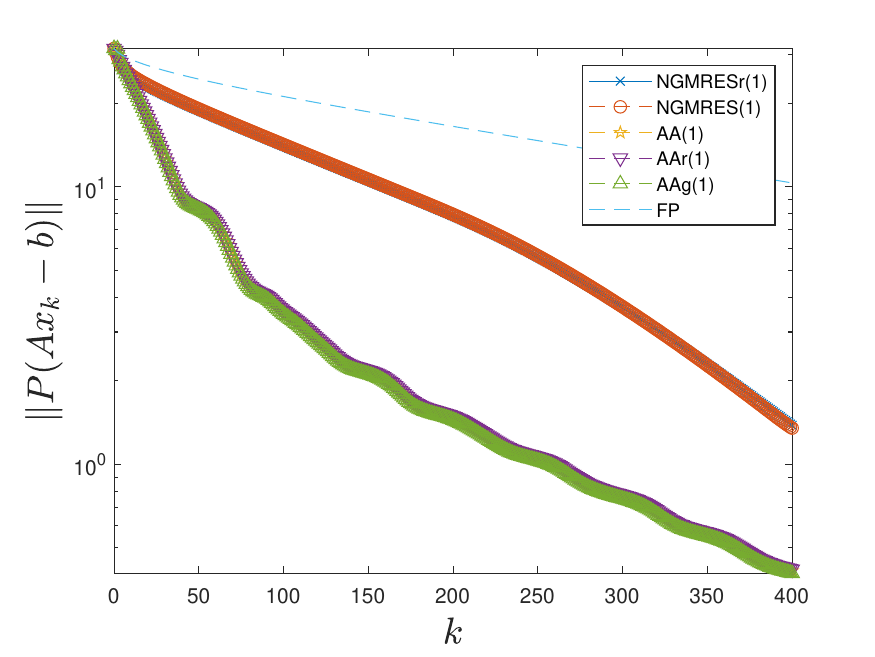}
		\includegraphics[width=0.3\textwidth]{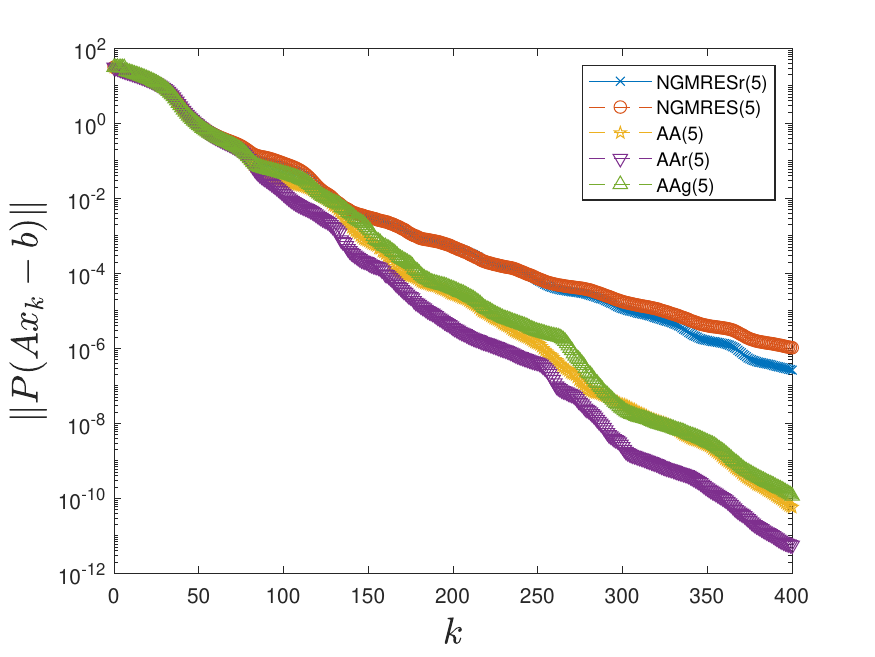}
        \includegraphics[width=0.3\textwidth]{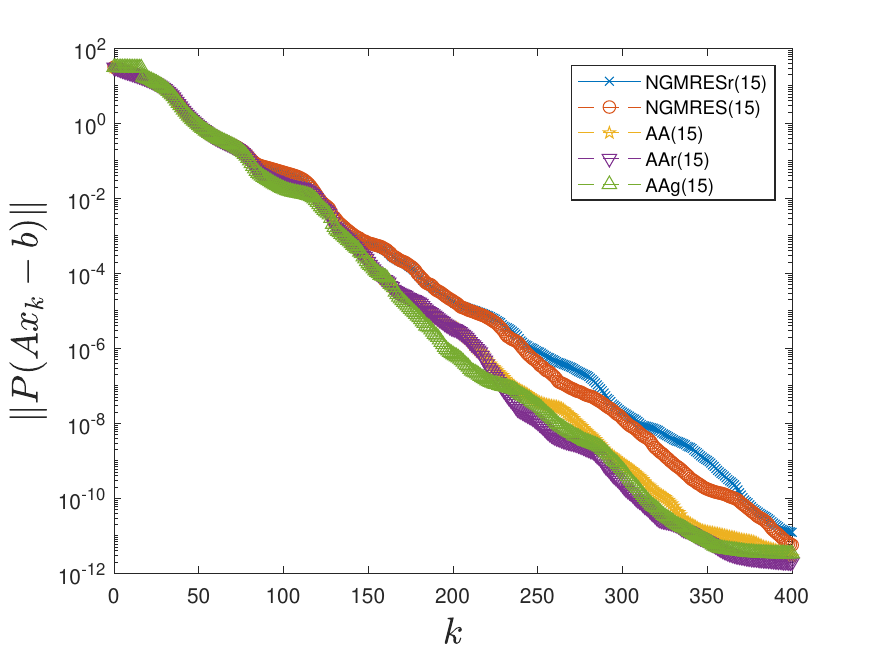}
		\caption{Example \ref{ex:Laplace} with $N=64$. Convergence history (preconditioned residual norms) for  windowed NGMRES-type and AA-type methods using  $m=1$ (left), $m=5$ (middle) and $m=15$ (right).}\label{fig:m5m15-all-methods}
	\end{figure}
\subsection{Nonsymmetric case} 
\begin{example}\label{ex:conv-diff}
   We consider a two-dimensional convection-diffusion model equation in a unit square, given by
   \begin{align*}
         -\Delta u+\sigma_1 u_x+\sigma_2 u_y =& f,    \quad \text{in}\quad \Omega=(0,1)\times (0,1),\\
           u =&0,  \quad  \text{on}\quad \partial \Omega.
     \end{align*}
\end{example}
We use a centered second-order five-point finite difference scheme for the Laplacian, and a centered second-order schemes for the first partial derivatives on an $N\times N$ grid with $h=1/(N+1)$,
scaling the matrix by $h^2$. We set $\frac{\sigma_1h}{2}=0.5$ and $\frac{\sigma_2h}{2}=0.5$. The right‑hand side $b$ is taken to be a vector of ones and the zero vector is used as the initial guess. The inverse of the preconditioner $P$ is the lower triangular part of matrix $A$. In figure \ref{fig:m-conv-diff}, we report the convergence history for fixed-point iteration, windowed NGMRES-type and AA-type methods using $m=1$ (left), $m=5$ (middle) and $m=15$. We see that AA-type and NGMRES-type methods greatly speed up the original fixed-point iteration. For $m=1,5$, windowed NGMRES and NGMRESr outperform AA-type methods especially for $m=1$, which is different than the symmetric case considered in section \ref{sec:symm}. When $m=15$, AA-type and NGMRES-type methods performs similarly. We notice that NGMRES(1) and NGMRESr(1) are much better than NGMRES(15) and NGMRESr(15). 

    \begin{figure}[H]
		\centering
		\includegraphics[width=0.3\textwidth]{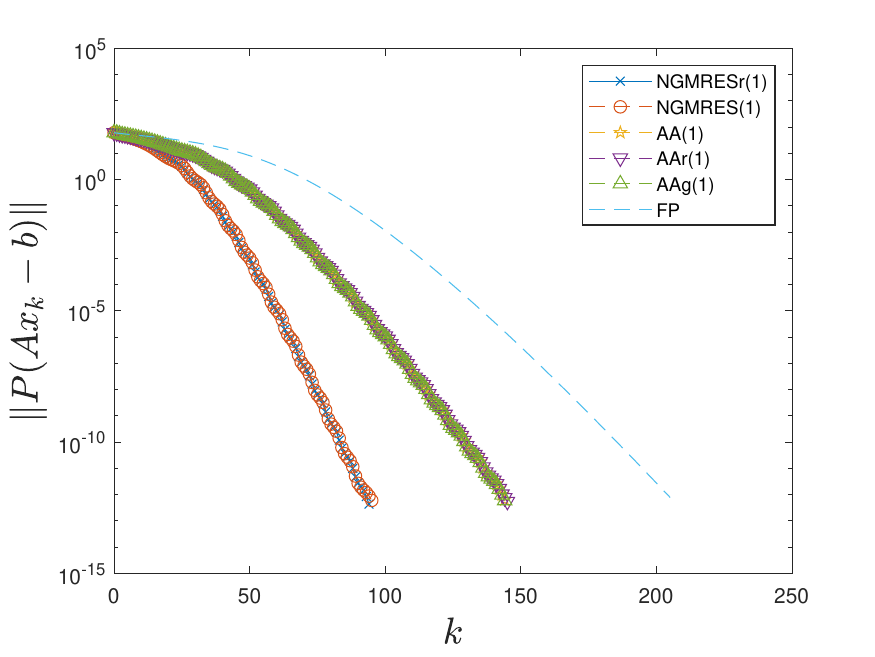}
        \includegraphics[width=0.3\textwidth]{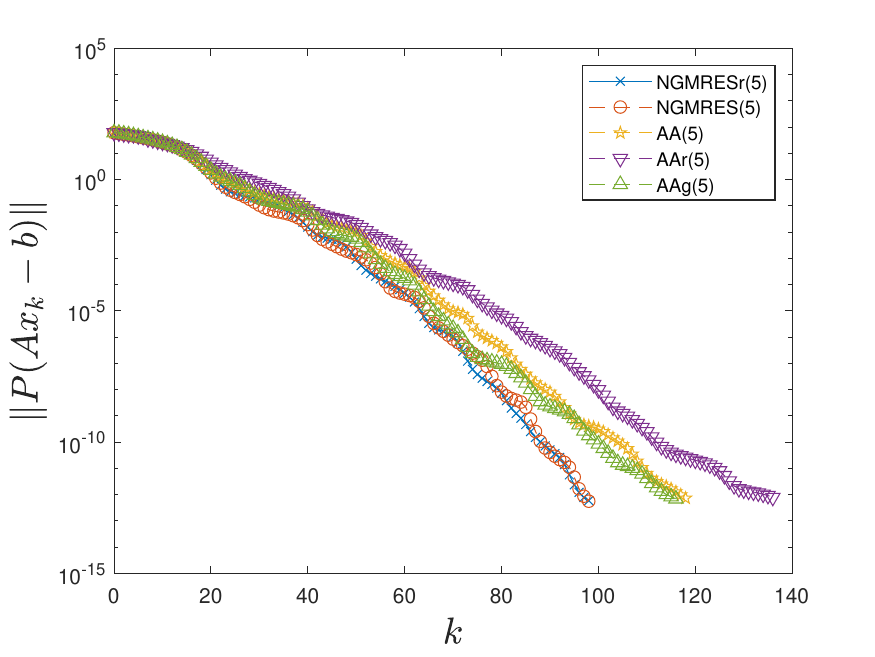}
        \includegraphics[width=0.3\textwidth]{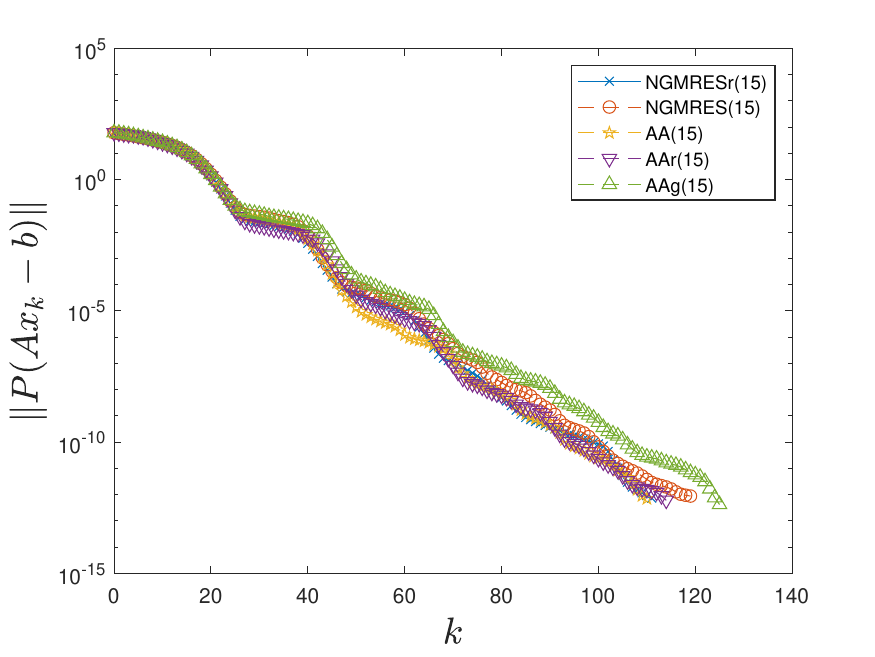}
		\caption{Example \ref{ex:conv-diff} with $N=64$. Convergence history (preconditioned residual norms) for windowed NGMRES-type and AA-type methods using $m=1$ (left), $m=5$ (middle) and $m=15$ (right).}\label{fig:m-conv-diff}
	\end{figure}
As mentioned earlier, the focus of this work is to understand the properties of NGMRES applied to the preconditioned Richardson iteration and its connection to preconditioned GMRES, as well as to provide guidance for practical users. However, the methods discussed here are general and can be applied to a broad class of nonlinear problems. We leave such investigations for future work, particularly those involving new variants of AA and NGMRES.
 \section{Conclusion}\label{sec:con}

In the work, We derive orthogonal properties of AA and NGMRES when applied to the preconditioned Richardson iteration. We further propose new variants of NGMRES and AA  using alternative least-squares formulations. The new least-squares formulation of NGMRES minimizes the preconditioned residual, and these of AA  when applied to the preconditioned Richardson iteration minimize the residual and the preconditioned residual, respectively. We establish connections of these variants with full version and windowed version  with preconditioned GMRES. We show that applying full NGMRES to the preconditioned Richardson iteration yields an algorithm equivalent to right-preconditioned GMRES, whereas the newly developed full NGMRES formulation corresponds to the left-preconditioned GMRES. Under certain conditions on the preconditioned coefficient matrix, we show the equivalence between preconditioned GMRES and windowed NGMRES.
Our theoretical findings deepen our understanding of AA‑type and NGMRES‑type methods for linear systems and offer guidance for their practical use. To evaluate the performance of the new variants, we present numerical experiments on linear problems.

More studies are needed to carefully assess these two types of methods in terms of implementation, CPU time, and performance on more general problems, especially for nonlinear problems, which we leave for future work. In addition, we plan to pursue the convergence analysis of the proposed variants for nonlinear problems and to extend their applications to highly challenging nonlinear systems in combination with other techniques, such as preconditioning.

\bibliographystyle{plain}  
\bibliography{AANGLSPbib}

\end{document}